\title{Moments of Two-Variable Functions and the Uniqueness
of Graph Limits}

\date{December 2008}

\documentclass[12pt]{article}
\usepackage{amssymb,amsmath,theorem}
\usepackage{hyperref}

\newtheorem{theorem}{Theorem}[section]

\newtheorem{lemma}[theorem]{Lemma}
\newtheorem{claim}{Claim}
\newtheorem{corollary}[theorem]{Corollary}
\theorembodyfont{\rmfamily}

\newtheorem{example}{Example}
\newenvironment{proof}{\medskip\noindent{\bf Proof. }}{\hfill$\square$\medskip}
\begin{document}

\def\R{{\mathbb R}}
\def\one{{\mathbf 1}}
\def\Q{{\mathbb Q}}
\def\Z{{\mathbb Z}}
\def\C{{\mathbb C}}
\def\N{{\mathbb N}}
\def\hom{{\rm hom}}
\def\Hom{{\rm Hom}}
\def\Inj{{\rm Inj}}
\def\Ind{{\rm Ind}}
\def\inj{{\rm inj}}
\def\ind{{\rm ind}}
\def\id{{\rm id}}
\def\sur{{\rm sur}}
\def\PAG{{\rm PAG}}
\def\io{{1\to\infty}}
\def\CUT{{\text{\rm CUT}}}
\def\eps{\varepsilon}
\long\def\killtext#1{}
\def\dd{{d_{\text{\rm set}}}}
\def\dl{{d_{\text{\rm left}}}}
\def\dr{{d_{\text{\rm right}}}}

\def\ontop#1#2{\genfrac{}{}{0pt}{}{#1}{#2}}

\def\AA{{\cal A}}\def\BB{{\cal B}}\def\CC{{\cal C}}\def\DD{{\cal D}}\def\EE{{\cal E}}\def\FF{{\cal F}}
\def\GG{{\cal G}}\def\HH{{\cal H}}\def\II{{\cal I}}\def\JJ{{\cal J}}\def\KK{{\cal K}}\def\LL{{\cal L}}
\def\MM{{\cal M}}\def\NN{{\cal N}}\def\OO{{\cal O}}\def\PP{{\cal P}}\def\QQ{{\cal Q}}\def\RR{{\cal R}}
\def\SS{{\cal S}}\def\TT{{\cal T}}\def\UU{{\cal U}}\def\VV{{\cal V}}\def\WW{{\cal W}}\def\XX{{\cal X}}
\def\YY{{\cal Y}}\def\ZZ{{\cal Z}}

\def\bA{{\mathbf A}}
\def\bB{{\mathbf B}}

\def\Prob{{\sf Prob}}
\def\Pr{{\sf P}}
\def\Qr{{\sf Q}}
\def\E{{\sf E}}
\def\Var{{\sf Var}}
\def\T{^{\sf T}}

\def\wt{\widetilde}
\def\tr{{\rm tr}}
\def\cost{\hbox{\rm cost}}
\def\val{\hbox{\rm val}}
\def\rk{{\rm rk}}
\def\diam{{\rm diam}}
\def\Ker{{\rm Ker}}

\author{Christian Borgs\\
Jennifer Chayes\\[6pt]
Microsoft Research, Cambridge, MA\\
L\'aszl\'o Lov\'asz\\[6pt]
Institute of Mathematics, E\"otv\"os Lor\'and University, Budapest, Hungary
}

\maketitle

\begin{abstract}
For a symmetric bounded measurable function $W$ on $[0,1]^2$ and a
simple graph $F$, the homomorphism density
\[
t(F,W)=\int_{[0,1]^{V(F)}} \prod_{ij\in E(F)} W(x_i,x_j)\,dx\,.
\]
can be thought of as a ``moment'' of $W$. We prove that
every such function is determined by its moments up to a measure
preserving transformation of the variables.

The main motivation for this result comes from the theory of
convergent graph sequences. A sequence $(G_n)$ of dense graphs is
said to be convergent if the probability, $t(F,G_n)$, that a random
map from $V(G_n)$ into $V(F)$ is a homomorphism converges for every
simple graph $F$. The limiting density can be expressed as $t(F,W)$
for a symmetric bounded measurable function $W$ on $[0,1]^2$. Our
results imply in particular that the limit of a convergent graph
sequence is unique up to measure preserving transformation.
\end{abstract}

\tableofcontents

\addtolength{\baselineskip}{3pt}
\section{Introduction}

Let $\WW$ be the set of bounded symmetric measurable functions
$W:~[0,1]^2\to\R$, and let $\WW_0$ denote the set of functions in
$\WW$ with values in $[0,1]$. For every $W\in \WW$ and every finite
graph $F$, we define the integral
\begin{equation} \label{tFW-def} t(F,W)=\int\limits_{[0,1]^n}
\prod_{ij\in E(F)} W(x_i,x_j)\,\prod_{i\in V(F)}dx_i\,.
\end{equation}
Our interest in these integrals stems from graph theory (see next
paragraph), but such integrals appear in physics, statistics, and
other areas. In many respects, these integrals can be thought of as
2-variable analogues of moments of 1-variable functions, so instead
of moment sequences, such $2$-variable functions have a "moment graph
parameter" (function defined on graphs). Just like moments of a
1-variable function determine the function up to measure preserving
transformations, these ``moments'' determine the 2-variable function
up to measure preserving transformations. The exact formulation and
proof of this fact is the main goal of this paper.

Our main motivation for this study comes from the theory of
convergent graph sequences. Let $F$ and $G$ be two simple graphs
(graphs without loops and multiple edges). Let us map the nodes of
$F$ randomly into $V(G)$, and let $t(F,G)$ denote the probability
that this map preserves adjacency. For example, $t(K_2,G)$ denotes
the edge density of $G$. In general, we call $t(F,G)$ the {\it
homomorphism density} or simply the {\it density} of $F$ in $G$.

We call a sequence of simple graphs $(G_n)$ {\it convergent}, if
$t(F,G_n)$ has a limit for every simple graph $F$. The notion of
convergent graph sequences was introduced by Borgs, Chayes, Lov\'asz,
S\'os and Vesztergombi \cite{BCLSV-unpub}, see also \cite{BCLSV-rev},
and further studied in \cite{dense1} and \cite{dense2}. Lov\'asz and
Szegedy \cite{LSz} proved that every convergent graph sequence has a
``limit object'' in the form of a function $W\in\WW_0$ in the sense
that
\begin{equation}\label{GTOW}
t(F,G_n)\longrightarrow t(F,W)\qquad \text{as} \quad n\to\infty
\end{equation}
for every simple graph $F$. In this case we say that {\it $G_n$
converges to} $W$. It was also shown in \cite{LSz} that for every
function $W\in\WW_0$ there is a convergent sequence $(G_n)$ of simple
graphs converging to $W$. To complete the picture, the results in
this paper imply that {\it the limit object is unique up to measure
preserving transformations}.

\section{Results}

For the precise statement of our results, we need some definitions.
Instead of the interval $[0,1]$, we consider two-variable functions
on an arbitrary probability space; while this does not add real
generality it leads to a cleaner picture.  We need a few definitions.

We start by recalling some basic notions from probability theory. Let
$(\Omega,\AA,\pi)$ be a probability space (where $\Omega$ is the
underlying set, $\AA$ is a $\sigma$-algebra on $\Omega$, and $\pi$ is
a probability measure on $\AA$). As usual, $(\Omega,\AA,\pi)$ is
called complete if $\AA$ contains all sets of external measure $0$,
and the completion of $(\Omega,\AA,\pi)$ is obtained by replacing
$\AA$ with the $\sigma$-algebra generated by $\AA$ and all subsets
$N\subset \Omega$ of external measure $0$.

Let $(\Omega,\AA,\pi)$ and $(\Omega',\AA',\pi')$ be probability
spaces, and let $\phi$ be a measure preserving map from $\Omega$ to
$\Omega'$. The map $\phi$ is called an {\it isomorphism} if it is a
bijection between $\Omega$ and $\Omega'$ and both $\phi$ and
$\phi^{-1}$ are  measure preserving, and it is called an {\it
isomorphism mod $0$} if there are null sets $N\in\AA$ and $N'\in\AA'$
such that the restriction of $\phi$ to $\Omega\setminus N$ is an
isomorphism between $\Omega\setminus N$ and $\Omega'\setminus N'$
(equipped with the suitable restrictions of $(\AA,\pi)$ and
$(\AA',\pi')$, respectively). In the last case $(\Omega,\AA,\pi)$ and
$(\Omega',\AA',\pi')$ are called {\it isomorphic mod $0$}.

It turns out that several of our results require a little bit more
structure than that of an arbitrary probability space.  In
particular, we will consider Lebesgue spaces, i.e., complete
probability spaces that are isomorphic mod $0$ to the disjoint union
of a closed interval (equipped with the standard Lebesgue sets and
Lebesgue measure) and a countable set of atoms.%
\footnote{See \cite{Roh}, Section 2.2 for an axiomatic definition of
Lebesgue spaces, and Section 2.4 for the proof that a probability
space is Lebesgue if and only if it is isomorphic mod 0 to the
disjoint union of a closed interval and a countable set of atoms.}

\subsection{Graphons and Graph Densities}
\label{sec:Graphons}

We are now ready to introduce the main objects studied in
this paper.

Starting from an arbitrary probability space $(\Omega,\AA,\pi)$, let
$W:~\Omega\times\Omega\to\R$ be a bounded, symmetric function
measurable with respect to the completion of
$(\Omega\times\Omega,\AA\times\AA,\pi\times\pi)$. We call the
quadruple $H=(\Omega,\AA,\pi,W)$ a {\it graphon}, and refer to $W$ as
{\it a graphon on the probability space $(\Omega,\AA,\pi)$.} (As
discussed above, such functions can be thought of as limits
convergent graph sequences, which explains the name).

From our point of view, graphons obtained by changing $W$ on a set of
measure $0$, or changing the $\sigma$-algebra $\AA$ so that $W$
remains measurable, do not differ essentially from the original.
However, for technical reasons we have to distinguish them. We say
that a graphon is {\it strong}, if $W$ is measurable with respect to
$\AA\times\AA$ (not just the completion of it). We can always change
$W$ on a set of measure $0$ to make the graphon strong (Theorem
\ref{thm:REDUCE}(i)).

We say that $H$ is {\it complete}, if the underlying probability
space is complete, and we say that it is {\it Lebesguian}, if the
underlying probability space is a Lebesgue space. The {\it
completion}, $\overline H$, of $H$ is obtained by completing the
underlying probability space, i.e., by replacing $\AA$ by its
completion $\overline\AA$.

Let $H=(\Omega,\AA,\pi,W)$ be a  graphon, and let $F$ be a finite
graph with $V(F)=\{1,\dots,k\}$. The definition \eqref{tFW-def} then
can be extended as
\begin{equation}\label{T-DEF}
t(F,H)=\int\limits_{\Omega^k} \prod_{ij\in E(F)}
W(x_i,x_j)\,\prod_{i=1}^k d\pi(x_i).
\end{equation}
Let $H=(\Omega,\AA,\pi,W)$ and $H'=(\Omega',\AA',\pi',W')$ be two
graphons. The goal of this paper is to determine necessary and
sufficient conditions under which
\begin{equation}\label{MAINCOND}
t(F,H)=t(F,H')
\end{equation}
for all graphs $F$.

To this end, we will introduce two different notions of isomorphism.
Both will be expressed in terms of the following operation: given a
graphon $H'=(\Omega',\AA',\pi',W')$ and a measure preserving map
$\phi$ from a probability space $(\Omega,\AA,W)$ into
$(\Omega',\AA',\pi')$, let $(W')^\phi$ be ``pull-back'' of $W'$,
defined by $(W')^\phi(x,y)=W(\phi(x),\phi(y))$.
If $H=(\Omega,\AA,\pi,W)$ and $G=(\Gamma,\BB,\rho,U)$ are two
graphons and $\phi:~\Omega\to Gamma$ is measure preserving from the
completion $\overline{\AA}$ into $\BB$ such that $W=U^\phi$ almost
everywhere, then we call $\phi$ a {\it weak isomorphism} from $H$ to
$G$. Note that a weak isomorphism is not necessarily invertible.

We say that $H$ and $H'$ are {\it isomorphic mod 0} (in notation
$H'\cong H'$), if there exists a map $\phi:~\Omega\to\Omega'$ such
that $\phi$ is an isomorphism mod 0 and $(W')^\phi=W$ almost
everywhere in $\Omega\times\Omega$. For simplicity, we often drop the
qualifier mod 0.

We call $H$ and $H'$ {\it weakly isomorphic} if there is a third
graphon $G$ and weak isomorphisms from $H$ and $H'$ into $G$. It will
follow from Theorems \ref{thm:REDUCE} and \ref{thm:Main} that we
could require here that $G$ is a strong Lebesguian graphon.

The isomorphism relation $\cong$ is clearly an equivalence relation,
and it will follow from Theorem~\ref{thm:Main} (ii) below that weak
isomorphism is an equivalence relation as well. Every graphon is
weakly isomorphic with its completion, and every pair of isomorphic
graphons is weakly isomorphic. It is clear that if two graphons $H$
and $H'$ are weakly isomorphic then \eqref{MAINCOND} holds for every
graph $H$. Theorem~\ref{thm:Main} (ii) below will show that the
converse also holds.

To state our results, we need one more notion, the notion of {\it
twins}.  Let $H=(\Omega,\AA,\pi,W)$ be a graphon. Two points
$x_1,x_2\in\Omega$ are called twins if $W(x_1,y)=W(x_2,y)$ for almost
all $y\in\Omega$. Note that relation of being twins is an equivalence
relation. We call the graphon $H$ {\it almost twin-free} if all there
exists a set $N$ of measure zero such that no two points in
$\Omega\setminus N$ are twins.

\subsection{Main results}\label{sec:MAIN}

With these definitions, we can state our main result:

\begin{theorem}\label{thm:Main}
{\rm (i)} If $H$ and $H'$ are almost twin-free Lebesguian graphons,
then \eqref{MAINCOND} holds for every simple graph $F$ if and only if
$H\cong H'$.

\smallskip

{\rm (ii)} If  $H$ and $H'$ are general graphons, then
\eqref{MAINCOND} holds for every simple graph $F$ if and only if $H$
and $H'$ are weakly isomorphic.
\end{theorem}

A natural idea of the proof of Theorem \ref{thm:Main} is the
following: can we bring a graphon $(\Omega,\AA,\pi,W)$ to a
``canonical form'', so that isomorphic or weakly isomorphic graphons
would have identical canonical forms? In the case of functions in a
single variable, this is possible, through ``monotonization'': for
every bounded real function on $[0,1]$ there is an unique monotone
increasing left-continuous function on $[0,1]$ that has the same
moments.

In Section \ref{sec:Canonical} we'll construct not quite a canonical
form, but a ``canonical ensemble'', a probability distribution
$(H_\alpha)$ of graphons on the same $\sigma$-algebra such that
$H\cong H_\alpha$ for almost all $\alpha$, and two graphons are
isomorphic if and only if their ensembles can be coupled so that
corresponding graphons are identical (up to sets of measure $0$).

An important element of the proof is a curious measure-theoretic
fact. Consider a 2-variable function for which all 1-variable
functions obtained by fixing one of the variables are measurable.
This of course does not in general imply that the 2-variable function
is measurable, but it does imply it in some circumstances (see e.g.
Corollary \ref{L-MEAS-1}).

As we will see, the second statement of Theorem \ref{thm:Main} can
easily be deduced from the first. In fact, we'll show that {\it every
graphon is weakly isomorphic to a twin-free Lebesguian graphon.} (See
Theorem \ref{thm:REDUCE} for more details of this isomorphism.)

We can also transform a Lebesguian graphon into a graphon whose
underlying probability space is the unit interval with the Lebesgue
measure, by ``resolving'' the atoms into intervals of the appropriate
length. This form is the most elementary and therefore useful in
applications; however, it is not so convenient for the purposes of
this paper because we loose twin-freeness.

It is easy to see that if $H$ and $H'$ are weakly isomorphic, then
\eqref{MAINCOND} holds not only for simple graphs $F$ but also for
graphs with multiple edges (which we'll call {\it multigraphs} if we
want to emphasize that multiple edges are allowed; but we don't allow
loops). Thus \eqref{MAINCOND} for simple graphs implies this equation
for multigraphs. (This fact will be an important step in the proof,
see Section \ref{MULTIPLE}.)

We can formulate our results in a probabilistic way. Recall that a
{\it coupling} between two probability spaces $(\Omega,\AA,\pi)$ and
$(\Omega',\AA',\pi')$ is a probability distribution on
$\AA\times\AA'$ whose marginals are $\pi$ and $\pi'$, respectively. A
coupling between two graphons means a coupling between their
underlying probability spaces. Let $H=(\Omega,\AA,\pi,W)$ be a
graphon, and let $X_1,\dots,X_n$ be independent random samples from
$\pi$. Then we have
\[
t(F,H)=\E\bigl(\prod_{ij\in E(F)} W(X_i,X_j)\bigr).
\]
Let $H=(\Omega,\AA,\pi,W)$ and $H'=(\Omega',\AA',\pi',W')$ two
graphons, and suppose that there exists a coupling $\gamma$ between
them such that $W(X_1,Y_1)=W'(X_2,Y_2)$ holds with probability $1$
for two independent samples $(X_1,X_2)$ and $(Y_1,Y_2)$ from
$\gamma$. In this case clearly (\ref{MAINCOND}) holds for every graph
$F$. As we will see, Theorem \ref{thm:Main} implies that in the
Lebesguian case the converse also holds.

We sum up the results for the most important special case of
functions in $\WW$, i.e., bounded, symmetric functions $W:[0,1]^2\to
\R$ which are measurable with respect to the Lebesgue sets on
$[0,1]^2$ (the Corollary would remain valid for arbitrary Lebesguian
graphons, but this would not be essentially more general).

\begin{corollary}\label{cor:UNIQUENESS}
For two functions $W,W'\in \WW$ the following are equivalent.

\smallskip

{\rm (a)} For every simple graph $F$, $t(F,W')=t(F,W)$.

\smallskip

{\rm (b)} For every multigraph $F$, $t(F,W')=t(F,W)$.

\smallskip

{\rm (c)} There exists a function $U\in\WW$ and two measure
preserving maps $\varphi,\psi:~[0,1]\to[0,1]$ such that
$W=U^{\varphi}$ and $W'=U^{\psi}$ almost everywhere.

\smallskip

{\rm (d)} There exist two measure preserving maps
$\varphi,\psi:~[0,1]\to[0,1]$ such that $(W')^\varphi=W^\psi$ almost
everywhere.

\smallskip

{\rm (e)} There exists a probability measure $\gamma$ on
$[0,1]\times[0,1]$ such that each marginal of $\gamma$ is the
Lebesgue measure, and if $(X,X')$ and $(Y,Y')$ are two independent
samples from $\gamma$, then $W(X,Y)=W'(X',Y'))$ with probability $1$.
\end{corollary}

\subsection{Examples}

The property of being twin-free is crucial for
Theorem~\ref{thm:Main} (i).

\begin{example}\label{NONTWIN}
Let $\phi_k:~[0,1]\to[0,1]$ be the map $\phi_k(x)=kx\pmod 1$. For any
function $W\in\WW$, the functions $W^{\phi_2}$ and $W^{\phi_3}$
define graphons that are weakly isomorphic but in general not
isomorphic. Indeed, for a ``generic'' $W$ (say $W=xy$), every point
has two twins in $W^{\phi_2}$ and three twins in $W^{\phi_3}$. The
pair of maps in Corollary \ref{cor:UNIQUENESS} (c) go from $W$, while
in (d), they go into
$(W^{\phi_3})^{\phi_2}=(W^{\phi_2})^{\phi_3}=W^{\phi_6}$.
\end{example}

Our next example shows that the Lebesgue property is also needed.

\begin{example}\label{NONLEB}
Let $\Omega$ be a subset of $[0,1]$ with inner Lebesgue measure $0$
and outer Lebesgue measure $1$, and let $\Omega'$ be its complement.
Let $\AA$ and $\AA'$ consist of the traces of Lebesgue measurable
sets on $\Omega$ and $\Omega'$, respectively. Let $W$ and $W'$ be the
restrictions of the function $xy$ to $\Omega\times\Omega$ and
$\Omega'\times\Omega'$, respectively. The identical embeddings
$\varphi:~\Omega\to[0,1]$ and $\varphi':~\Omega'\to[0,1]$ are measure
preserving, and hence $H=(\Omega,\AA,\pi,W)$ and
$H'=(\Omega',\AA',\pi',W')$ are weakly isomorphic. But for every
$x\in\Omega$, we have
\[
2\int\limits\limits_{\Omega} W(x,y)\,d\pi(y) = x\notin \Omega',
\]
which shows that there is no way to ``match up'' the points in
$\Omega$ and $\Omega'$ to get an isomorphism mod $0$. The same
example shows that conclusions (d), (e) in Corollary
\ref{cor:UNIQUENESS} could not be extended to the non-Lebesgue case
either.
\end{example}

\section{Isomorphism}\label{sec:Reduc}

The main goal of this section is to describe how a general graphon
can be transformed into a twin-free Lebesguian graphon. To this end,
we have to recall some basic notions from measure theory (mostly
because their usage does not seem standard), and then discuss
different ``isomorphism-like'' mappings between graphons.

\subsection{Preliminaries}

For a set $\SS$ of subsets of a set $\Omega$, we denote by
$\sigma(\SS)$ the $\sigma$-algebra generated by $\SS$. We call a
$\sigma$-algebra $\AA$ {\it countably generated} if there is
countable set $S\subseteq\AA$ such that $\sigma(\SS)=\AA$. This is
equivalent to the existence of a sequence
$\AA_1\subseteq\AA_2\subseteq\dots$ of finite $\sigma$-algebras whose
union generates $\AA$.

We say that a set $\SS\subseteq\AA$ is a {\it basis} for the
probability space $(\Omega,\AA,\pi)$, if $\sigma(S)$ is dense in
$\AA$, i.e., for every $X\in\AA$ there is a $Y\in\sigma(\SS)$ such
that $\pi(X\triangle Y)=0$.

Given sets $A\subset\Omega$ and two points $x,y\in\Omega$, we say
that $A$ {\it separates} $x$ and $y$ if $|\{x,y\}\cap A|=1$.  We say
that a set $\SS$ of subsets of $\Omega$ {\it separates} $x$ and $y$
if there exists a set $A\in\SS$ that separates $x$ and $y$. This
leads to a partition $\PP[\SS]$ of $\Omega$ by placing two points in
the same class if and only if they are not separated by $\SS$. We say
that $\SS$ is {\it separating} if it separates any two points in
$\Omega$. We'll say that a graphon is separating if its underlying
$\sigma$-algebra is separating.

A probability space $(\Omega',\AA',\pi')$ is called a {\it full
subspace} of $(\Omega,\AA,\pi)$ if $\Omega'$ is a (not necessarily
measurable) subset of $\Omega$ of external measure $1$,
$\AA'=\{A\cap\Omega'\mid A\in\AA\}$, and $\pi'(A\cap\Omega'))=\pi(A)$
for all $A\in\AA$.

Consider two probability spaces $(\Omega,\AA,\pi)$ and
$(\Omega',\AA',\pi')$ and a measure preserving map
$\phi:\Omega\to\Omega'$. The map $\phi$ is called an {\it embedding}
of the first space into the second if $\phi$ is an isomorphism
between $(\Omega,\AA,\pi)$ and a full subspace of
$(\Omega',\AA',\pi')$. We call $\phi$ an {\it embedding} of a graphon
$H=(\Omega,\AA,\pi,W)$ into a graphon $H'=(\Omega',\AA',\pi',W')$ if
$\phi$ is an embedding of $(\Omega,\AA,\pi)$ into
$(\Omega',\AA',\pi')$ and $(W')^{\phi}=W$ almost everywhere.

Let $(\Omega,\AA,\pi)$ be a probability space and $f:~\Omega\to\R$, a
bounded $\AA$-measurable function. Let $\AA_0\subseteq\AA$ be a
sub-$\sigma$-algebra. The {\it conditional expectation}
$\E(f\mid\AA_0)$ is the set of all $\AA_0$-measurable function $f'$
such that $\int_{A_0} f\,d\pi=\int_{A_0}f'\,d\pi$ for all
$A_0\in\AA_0$. It is well known that such functions exist and any two
such functions differ only on a set of $\pi$-measure $0$. We'll write
(somewhat sloppily) $f'=\E(f\mid\AA_0)$ instead of
$f'\in\E(f\mid\AA_0)$. We say that $f$ is {\it almost
$\AA_0$-measurable}, if there is an $\AA_0$-measurable function $f'$
such that $f=f'$ $\pi$-almost everywhere. Clearly we must have
$f'\in\E(f\mid\AA_0)$, and it does not matter which representative of
$\E(f\mid\AA_0)$ we choose, so (again somewhat sloppily) we can say
that $f$ is almost $\AA_0$-measurable if and only if
$f=\E(f\mid\AA_0)$ almost everywhere.

\subsection{Push-Forward and Quotients}
\label{sec:W-phi}

Let $(\Omega,\AA,\pi)$ and $(\Omega',\AA',\pi')$ be probability
spaces and let $\phi:~\Omega\to\Omega'$ be a measure preserving map.
We have described how to ``pull back'' a graphon on
$(\Omega',\AA',\pi')$ to a (weakly isomorphic) graphon on
$(\Omega,\AA,\pi)$. It is also possible to ``push-forward'' a graphon
$H=(\Omega,\AA,\pi,W)$ to a graphon $(\Omega',\AA',\pi',W_\phi)$.
This is defined by the requirement that
\begin{equation}
\label{W-phi-def} \int_{A_1'\times
A_2'}W_\phi(x',y')d\pi'(x')d\pi'(y')
=\int\limits_{\phi^{-1}(A_1')\times\phi^{-1}(A_2')}
W(x,y)\,d\pi(x)\,d\pi(y)
\end{equation}
for all $A_1',A_2'\in\AA'$. The next lemma states that the
``push-forward'' $W_\phi$ is well defined, and that $(W_\phi)^\phi$
is a certain conditional expectation of $W$.

\begin{lemma}
\label{lem:W-Phi-Phi} Let $(\Omega,\AA,\pi)$ and
$(\Omega',\AA',\pi')$ be probability spaces, let
$\phi:\Omega\to\Omega'$ be a measure preserving map, and let $W$ be a
graphon on $(\Omega,\AA,\pi)$.

\smallskip

{\rm (i)} There exists a bounded, symmetric function
$W_\phi:\Omega'\times\Omega'\to\R$ that is $\AA'\times\AA'$
measurable and satisfies \eqref{Rad-Nik-def}.  It is unique up to
changes on a set of measure zero in $\Omega'\times\Omega'$.

\smallskip

{\rm (ii)} Let $\AA_\phi=\phi^{-1}(\AA')$. Then $(W_\phi)^\phi=\E
(W\mid \AA_\phi\times\AA_\phi)$ almost everywhere.

\smallskip

{\rm (iii)} If $\phi$ is an embedding of $(\Omega,\AA,\pi)$ into
$(\Omega',\AA',\pi')$, then $(W_\phi)^\phi=W$ almost everywhere.
\end{lemma}

\begin{proof}
(i) By linearity, it is easy to see that we can restrict ourselves to
the case where $W$ takes values in $[0,1]$. Define a measure $\mu$ on
$\AA'\times\AA'$  by
\[
\mu(A_1'\times A_2')
=\int\limits_{\phi^{-1}(A_1')\times\phi^{-1}(A_2')}
W(x,y)\,d\pi(x)\,d\pi(y)
\]
for $A_1',A_2'\in\AA$.  With this definition, we have that
\[
0\leq\mu(A_1'\times A_2')\leq
\pi(\phi^{-1}(A_1'))\pi(\phi^{-1}(A_2'))= (\pi'\times\pi')(A_1'\times
A_2'),
\]
implying in particular that $\mu$ is absolutely continuous with
respect to $\pi'\times\pi'$.  Hence the Radon-Nikodym derivative,
\begin{equation}
\label{Rad-Nik-def} W_\phi=\frac{d\mu}{d(\pi'\times\pi')},
\end{equation}
is well defined. Using the above bound once more, together with the
fact that $\mu(A_1\times A_2)=\mu(A_2\times A_1)$, we furthermore
have that
\begin{equation}
\label{W-Phi-Bds} 0\leq W_\phi(x,y)\leq 1 \qquad\text{and}\qquad
W_\phi(x,y)=W_\phi(y,x)
\end{equation}
almost everywhere. Changing $W_\phi$ on a set of measure zero, we may
assume that these relations hold everywhere. To define $W_\phi$ for a
general bounded function $W$, we  use linearity.

(ii) Let $A_1, A_2\in \AA_\phi$, i.e., let $A_1=\phi^{-1}(A_1')$ and
$A_2=\phi^{-1}(A_2')$ for some $A_1',A_2'\in\AA'$. By the definition
of $W_\phi$, the fact that $\phi$ is measure preserving, and the
definition of $(W_\phi)^\phi$, we have that
\[
\begin{aligned}
\int\limits_{A_1\times A_2} W(x,y)\,d\pi(x)\,d\pi(y) &=
\int\limits_{A_1'\times A_2'} W_\phi(x',y')\,d\pi'(x')\,d\pi'(y')
\\
&= \int\limits_{A_1\times A_2}
W_\phi(\phi(x),\phi(y))\,d\pi(x)\,d\pi(y)
\\
&= \int\limits_{A_1\times A_2} (W_\phi)^\phi(x,y)\,d\pi(x)\,d\pi(y).
\end{aligned}
\]
This implies that $(W_\phi)^\phi=\E (W\mid \AA_\phi\times \AA_\phi)$
almost everywhere.

(iii) Since $\phi$ is an isomorphism between $(\Omega,\AA,\pi)$ and a
subspace of $(\Omega',\AA',\pi')$, we know that given any $A\in\AA$,
we can find an $A'\in\AA'$ such that $\phi(A)=A'\cap\phi(\Omega)$.
But then $\phi^{-1}(A')=\phi^{-1}(\phi(A))=A$, proving that
$A\in\AA_\phi$.  Thus $\AA_\phi=\AA$, which implies that
$(W_\phi)^\phi=W$ almost everywhere.
\end{proof}

We can use the ``push-forward'' construction to define quotients of
graphons. Let $H=(\Omega,\AA,\pi,W)$ be a graphon, let $\PP$ be an
arbitrary partition of $\Omega$ into disjoint sets, and for
$x\in\Omega$, let $[x]$ denote the class in $\PP$ that contains the
point $x$. We then define a graphon
$H/\PP=(\Omega/\PP,\AA/\PP,\pi/\PP,W/\PP)$ and a measure preserving
map $\phi:\Omega\to\Omega/\PP$ as follows: the points in $\Omega/\PP$
are the classes of the partition $\PP$, $\phi$ is the map
$\phi:x\mapsto [x]$, $\AA/\PP$ is the $\sigma$-algebra consisting of
the sets $A'\subset\Omega/\PP$ such that $\phi^{-1}(A')\in\AA$, and
$(\pi/\PP)(A'):=\pi(\phi^{-1}(A'))$. Then $\phi$ is measure
preserving, and the function $W/\PP=W_\phi$ is defined by
\eqref{W-phi-def}.

\subsection{Reductions}

Now we are able to state the theorem that allows us to reduce every
graphon to a twin-free Lebesguian graphon.

\begin{theorem}\label{thm:REDUCE}
{\rm (i)} Let $H=(\Omega,\AA,\pi,W)$ be a graphon.  Then one can
change the value of $W$ on a set of $\pi\times\pi$-measure $0$ to get
a strong graphon.

{\rm (ii)} Let $H=(\Omega,\AA,\pi,W)$ be a graphon.  Then there
exists a countably generated $\sigma$-algebra $\AA_0\subset\AA$ such
that $W$ is $(\AA_0\times\AA_0)$-measurable.

\smallskip

{\rm (iii)} Let $H=(\Omega,\AA,\pi,W)$ be a graphon. Then the graphon
$H/\PP[\AA]$ is separating. If $H$ is countably generated, then so is
$H/\PP[\AA]$.

\smallskip

{\rm (iv)} Let $H=(\Omega,\AA,\pi,W)$ be a separating graphon on a
probability space with a countable basis. Then the completion of $H$
can be embedded into a Lebesguian graphon.

\smallskip

{\rm (v)} Let $H=(\Omega,\AA,\pi,W)$ be a graphon, and let $\PP$ be
the partition into the twin-classes of $H$.  Then $H/\PP$ is almost
twin-free. If $H$ is Lebesguian, then $H/\PP$ is Lebesguian as well.
Furthermore, the projection $H\to H/\PP$ is a weak isomorphism.
\end{theorem}

\begin{corollary}\label{COR:ON-LEBESGUE}
Every graphon has a weak isomorphism into a strong Lebesguian
graphon.
\end{corollary}

The proof of this theorem (which is not hard, but technical) will be
given in the rest of this section.

\subsubsection{Making a graphon strong}

Let $H=(\Omega,\AA,\pi,W)$ be a graphon, and let
$W'=\E(W\mid\AA\times\AA)$. Then $W'$ is $\AA\times\AA$-measurable,
and changing $W'$ on a set of measure $0$, we may assume that $W'$ is
symmetric and bounded. Moreover, $\int_{A\times A'} (W'-W) = 0$ for
all $A,A'\in\AA$, which implies that $\int_S (W'-W) = 0$ for all sets
$S$ in the completion of $\AA\times\AA$, so $W=W'$ almost everywhere.
These observations prove part (i) of the Theorem.

\subsubsection{Countable generation}

We prove a simple lemma, which implies Theorem \ref{thm:REDUCE}(ii),
and will also be used at several other places
(Sections~\ref{sec:MARGINAL-MEASURE} and \ref{MULTIPLE}).

\begin{lemma}
\label{lem:Red-to-Count} Let $(\Omega,\AA)$ and $(\Omega',\AA')$ be
measurable spaces, and let $W:\Omega\times\Omega'\to\R$ be a bounded,
$(\AA\times\AA')$-measurable function. Then there exist countably
generated $\sigma$-algebras $\AA_0\subset\AA$ and $\AA_0'\subset\AA'$
such that $W$ is  $(\AA_0\times\AA'_0)$-measurable.
\end{lemma}

\begin{proof}
Let $\CC$ be the set of bounded, $(\AA\times\AA')$-measurable
functions $W$ for which the statement of the lemma is true. The set
$\CC$ is clearly a vector space that contains the constant function
$1$ as well as the indicator functions of all rectangles $A\times B$
with $A\in\AA$ and $B\in\AA'$.  If is further not hard to show that
if $(W_k)$ is a sequence of non-negative functions in $\CC$ and
$W_k\uparrow W$ for a bounded function $W$, then the limiting
function $W$ is in $\CC$ as well.  By the monotone class theorem
(see, e.g., Theorem 3.14 in \cite{Wil}), we conclude that $\CC$
contains all bounded  functions which are measurable with respect to
the $\sigma$-algebra generated by the rectangles  $A\times B$, i.e.,
the $\sigma$-algebra $\AA\times\AA'$.
\end{proof}

\subsubsection{Merging inseparable elements}

If we identify elements in the same class of the partition
$\PP[\AA]$, we get a $\sigma$-algebra which is isomorphic under the
obvious map. This implies (iii) of Theorem \ref{thm:REDUCE}.

\subsubsection{Lebesgue property}\label{pf:LEBESGUE}

Consider a separating graphon $H=(\Omega,\AA,\pi,W)$, and assume that
$\AA$ is generated by the countable set $\SS$. Then $\SS$ is a basis
for the completion of $(\Omega,\AA,\pi)$. We invoke the fact (see
e.g. \cite{Roh}, Section 2.2) that any separating complete
probability space with a countable basis can be embedded into a
Lebesgue space. Thus there exists an embedding $\psi$ of the
completion of $(\Omega,\AA,\pi)$ into a Lebesgue space
$(\Omega',\LL',\lambda')$. Let $W'$ be the push-forward  of $W$,
$W'=W_\psi$.  By Lemma~\ref{lem:W-Phi-Phi}, we have that
$(W')^\psi=W$ almost everywhere, which shows that $\psi$ is an
embedding of the completion of $H$ into $(\Omega',\LL',\lambda',W')$.
This proves part (iv) of Theorem \ref{thm:REDUCE}.

\subsubsection{Partitions into Twin-Classes}

We prove (v) in Theorem \ref{thm:REDUCE}. We may assume that $\AA$ is
countably generated. Indeed, by Lemma \ref{lem:Red-to-Count}, we can
replace $\AA$ by a countably generated $\sigma$-algebra $\AA_0$. This
does not change the relation of being twins: Two points
$x,x'\in\Omega$ are twins if and only if the set
$A_{x,x'}=\{y\in\Omega: W(x,y)=W(x,y')\}$ has measure $1$. Since $W$
is measurable with respect to $\AA_0\times\AA_0$, the set $A_{x,x'}$
lies in $\AA_0\subset\AA$, implying that $x$ and $x'$ are twins with
respect to $H$ if and only if they are twins with respect to $H_0$.

Let $\AA_\PP$ consists of those sets in $\AA$ that do not separate
any pair of twin points. Clearly $\AA_\PP$ is a $\sigma$-algebra.

\begin{claim}\label{Cl:WPW}
$W$ is almost $\AA_\PP\times\AA_\PP$-measurable.
\end{claim}

Let  $\widetilde W=\E (W\mid \AA_\PP\times\AA_\PP)$. We want to prove
that
\begin{equation}\label{cond-equal}
\int_{A\times B} W(x,y)d\pi(x) d\pi(y) =
\int_{A\times B} \widetilde W(x,y)d\pi(x) d\pi(y)
\end{equation}
for all $A,B\in\AA$. Define the functions
\[
g_A=\E(1_A\mid\AA_\PP),\qquad  U_A(y)=\int_A W(x,y) d\pi(x), \qquad
V_A(x)=\int W(x,y)g_A(y) d\pi(y).
\]
Since $U_A(y)=U_A(z)$ if $y,z$ are twins, the function $U_A$ is
$\AA_\PP$-measurable, similarly for $V_A$, and obviously for $g_A$.
Repeatedly using the fact that $\int fg = \int f\E(g\mid\AA_0)$ if
$f$ is $\AA_0$-measurable, this implies
\[
\begin{aligned}
\int_{A\times B} W(x,y)d\pi(x)d\pi(y) &= \int 1_B(y) U_A(y)d\pi(y)
=\int g_B(y) U_A(y) d\pi(y)
\\&= \int V_B(x) 1_A(x) d\pi(x)=\int V_B(x) g_A(x) d\pi(x)
\\&=\int W(x,y)g_A(x)g_B(y)d\pi(x)d\pi(y)
\\&=\int \widetilde W(x,y)g_A(x)g_B(y)d\pi(x)d\pi(y)
\\&= \int_{A\times B} \tilde W(x,y) d\pi(x) d\pi(y).
\end{aligned}
\]
(where the last equality follows since $\widetilde W$ is
$\AA_\PP\times\AA_\PP$-measurable). This implies \eqref{cond-equal}
and completes the proof of Claim \ref{Cl:WPW}.

Let $\widetilde W=\E (W\mid \AA_\PP\times\AA_\PP)$ as before, then
$H_\PP =(\Omega,\AA_\PP,\pi,\widetilde W)$ is a graphon, which is
clearly weakly isomorphic to $(\Omega,\AA,\pi,W)$.
Let $N$ be the set of points $x\in\Omega$ for which
$\{y\in\Omega:~\widetilde W(x,y)\not= W(x,y)\}$ has positive measure.
Then clearly $N$ is a null set, and two points
$x,x'\in\Omega\setminus N$ are twins in $H$ if and only if they are
twins in $H_\PP$. The graphon $H/\PP$ is obtained from $H_\PP$ by
identifying indistinguishable elements, which implies that $H/\PP$ is
twin-free.

To prove that $H/\PP$ is  Lebesguian if $H$ is Lebesguian, we invoke
the fact (established in Section 3.2 of \cite{Roh}) that
$(\Omega/\PP,\AA/\PP,\pi/\PP)$ is a Lebesgue space provided
$(\Omega,\AA,\pi)$ is a Lebesgue space and there exists a countable
set $\SS\subseteq\AA$ that separates two points if and only they are
in different partition classes.

To construct such a set $\SS$, let $\TT$ be a countable set
generating $\AA$, closed under finite intersections.
For $A\in\AA$ and $x\in\Omega$, let
\[
\mu_x(A)=\int_AW(x,y)d\pi(y).
\]
Since $W$ is a bounded $\AA\times\AA$-measurable function, the
function $A\mapsto \mu_x(A)$ is a finite measure for all $x\in
\Omega$, while the function $x\mapsto \mu_x(A)$ is a $\AA$-measurable
function on $\Omega$ for all $A\in\AA$.

By definition, $x,x'\in\Omega$ are twins iff the set $\{y\in\Omega:
W(x,y)=W(x,y')\}$ has measure zero. This is equivalent to the
condition that $\mu_x(A)=\mu_{x'}(A)$ for all $A\in\AA$.  Since the
measure $\mu_x(\cdot)$ on $\AA$ is uniquely determined by the sets in
$\TT$, we have that $x$ and $x'$ are twins if and only if
$\mu_x(T)=\mu_{x'}(T)$ for all $T\in\TT$.

For every $T\in \TT$ and rational number $r$, consider the sets
$S_{T,r}=\{x\in\Omega:~\mu_x(T)\ge r\}$. There is a countable number
of these. Furthermore, if $x$ and $x'$ are twins, then they belong to
exactly the same sets $S_{T,r}$; if they are not twins, then there is
a $T\in\TT$ such that $\mu_x(T)\not=\mu_{x'}(T)$, and for any
rational number between $\mu_x(T)$ and $\mu_{x'}(T)$, the set
$S_{T,r}$ separates $x$ and $x'$.

This completes the proof of Theorem \ref{thm:REDUCE}.

\subsection{Isomorphism and Weak Isomorphism}

We conclude this section with relating isomorphism and weak
isomorphism.

\begin{lemma}\label{ISO-1}
Let $H_i=(\Omega_i,\AA_i,\pi_i,W_i)$ be graphons with the Lebesgue
property ($i=1,2$), and let $\phi:~\Omega_1\to\Omega_2$ be
measure-preserving. If $H_1$ is almost twin-free, and
$W_1=W_2^{\phi}$ almost everywhere, then $\phi$ is an isomorphism mod
$0$, so in particular $H_1\cong H_2$.
\end{lemma}

\begin{proof}
Let
\[
\Omega_1'=\{x\in \Omega_1\colon~ W_2(\phi(x),\phi(y))= W_1(x,y)
 \text{ for almost all }y\},
\]
and let $N_1=\Omega_1\setminus\Omega_1'$. Then $\pi_1(N_1)=0$ by
Fubini and our assumption that $W_1=W_2^{\phi}$ almost everywhere.

Let $N_1'$ be a nullset such that all twin-classes of $H_1$ have at
most one point in $\Omega_1\setminus N_1'$, and let $\phi'$ to be the
restriction of $\phi$ to $\Omega_1'\setminus N_1'$. Then $\phi'$ is
injective: indeed, if $x_1,x_2\in\Omega_1'\setminus N_1'$ and
$\phi(x_1)=\phi(x_2)$, then
$W_1(x_1,y)=W_2(\phi(x_1),\phi(y))=W_2(\phi(x_2),\phi(y))
=W_1(x_2,y)$ for almost all $y$ by the definition of $\Omega_1'$,
hence $x_1$ and $x_2$ are twins, a contradiction. As shown in
\cite{Roh}, Section 2.5, an injective measure preserving map between
Lebesgue spaces has a measurable inverse defined almost everywhere.
This implies that $\phi':\Omega_1'\setminus N_1'\to\Omega_2$ is an
isomorphism mod $0$, which shows that $\phi$ is an isomorphism mod
$0$ as well.
\end{proof}

\begin{corollary}\label{Cor:WEAK-LEB}
If two twin-free graphons with the Lebesgue property are weakly
isomorphic, then they are isomorphic.
\end{corollary}

\section{Canonical Ensembles} \label{sec:Canonical}

We could try to construct a ``canonical form'' of a graphon by
assigning ``tags'' to the points in $\Omega$. For example, we could
tag a point $x$ with its marginal $d(x)=\int W(x,y)\,d\pi(y)$, or by
the sequence of marginals of higher powers of $W$. This, however,
would not work: for example, there could be a transitive group of
measure-preserving permutations of $\Omega$ leaving $W$ invariant,
and then all points would still have the same tag.

To break the symmetry, we select an infinite sequence
$\alpha=(a_1,a_2,\dots)$ of points in $\Omega$, which we call {\it
anchor points.} Now we can tag each point $x\in\Omega$ with the
sequence
\begin{equation}\label{PHI-DEF}
\Phi_\alpha(x)=(W(x,a_1),W(x,a_2),\dots)\in [0,1]^\N
\end{equation}
(where we assume that $0\le W\le 1$) The map $x\mapsto
\Phi_\alpha(x)$ defines a measurable map from $\Omega$ into
$[0,1]^{\N}$ (with respect to the standard Borel $\sigma$-algebra
$\LL$ on $[0,1]^{\N}$), which in turn defines a measure
$\lambda_\alpha$ on the sets $S\in\LL$ by
\begin{equation}\label{LAMBDADEF}
\lambda_{\alpha}(S)=\pi(\Phi_\alpha^{-1}(S)),
\end{equation}
and a graphon $W_{\Phi_\alpha}$ on $([0,1]^{\N},\LL,\lambda_\alpha)$
by \eqref{W-phi-def}. We denote the completion of
$([0,1]^{\N},\LL,\lambda_\alpha,W_{\Phi_\alpha})$ by $H_\alpha$.

We will show that if $\alpha_1,\alpha_2,\dots$ are taken i.i.d. at
random with distribution $\pi$ then with probability one, then
$H_\alpha$ is isomorphic mod $0$ to the original graphon $H$ (see
Section~\ref{sec:anchor} for details). So using an infinite sequence
of independent random points as anchor points, the tags of the points
contain all information about the points.

These tags are almost canonical, except for the choice of the
sequence $\alpha$. So instead of a canonical form, we get a
``canonical ensemble'', a probability distribution $(H_\alpha)$ of
graphons such that $H\cong H_\alpha$ for almost all $\alpha$, and two
graphons are isomorphic if and only if their ensembles can be coupled
so that corresponding graphons are isomorphic.

To prove Theorem \ref{thm:Main} (i), we will therefore have to show
that if $H$ and $H'$ satisfy \eqref{MAINCOND}, then we can ``couple''
the choice of anchor points $\alpha$ in $H$ and $\beta$ in $H'$ so
that $H_\alpha\cong H'_\beta$, thus yielding an isomorphism of $H$
and $H'$. This second step in the proof will be carried out in
Section~\ref{sec:coupling}.

\subsection{Measure theoretic preparation}\label{sec:MARGINAL-MEASURE}

The next technical lemma will be important in the construction of
``canonical ensembles''.

\begin{lemma}\label{L-MEAS-2x}
Let $(\Omega,\AA,\pi)$ and $(\Omega',\AA',\pi')$ be probability
spaces, and let $W:\Omega\times\Omega'\to\R$ be a bounded
$\AA\times\AA'$-measurable function. Let $Y_1,Y_2,\dots$ be
independent random points from $\Omega'$. Let $\AA_0\subseteq \AA$ be
the (random) $\sigma$-algebra generated by the functions
$W(\cdot,Y_k)$. Then with probability $1$, $W$ is almost
$\AA_0\times\AA'$-measurable.
\end{lemma}

\begin{proof}
By Lemma~\ref{lem:Red-to-Count}, we may assume that $\AA$ and $\AA'$
are countably generated. Let $\AA'_1\subset\AA'_2\subset\dots$ and
$\AA'_1\subset\AA'_2\subset\dots$ be a sequence of finite
$\sigma$-algebras with $\sigma(\cup_n \AA_n)=\AA$ and $\sigma(\cup_n
\AA'_n)=\AA'$, and let $P_n'$ denote the partition of $\Omega'$ into
the atoms of $\AA'_n$. For $y\in S\in P_n'$ with $\pi'(S)>0$, define
\[
U_{n,m}(x,y)= \frac 1{m\pi(S)}\sum_{\ontop{j\le m}{Y_j\in S}}
W(x,Y_j)
\]
We define $U_{n,m}(x,y)=0$ if $y\in S\in P_n'$ with $\pi'(S)=0$.

First we prove that for every $n\ge 1$, every $A\in\AA$ and
$A'\in\AA'_n$, we have with probability $1$
\begin{equation}\label{EQ:UNW}
\int\limits_{A\times A'} U_{n,m}\,d\pi\,d\pi'\longrightarrow
\int\limits_{A\times A'} W\,d\pi\,d\pi' \qquad (m\to \infty).
\end{equation}
It suffices to prove this in the case when $A'=S\in P_n'$ and
$\pi'(S)>0$. Then for every $y_0\in A'$, we have
\[
\int\limits_A U_{n,m}(x,y_0)\,d\pi(x)= \frac
1{m\pi(S)}\sum_{\ontop{j\le m}{Y_j\in S}} \int\limits_A
W(x,Y_j)\,d\pi(x).
\]
hence by the Law of Large Numbers,
\[
\int\limits_A U_{n,m}(x,y_0)\,d\pi(x)\longrightarrow \frac
1{\pi(S)}\int\limits_{A\times S} W\,d\pi\,d\pi' \qquad (m\to \infty).
\]
Since both sides are independent of $y_0\in S$, integrating over
$y_0\in S$ equation \eqref{EQ:UNW} follows.

The number of choices of $n$, $A\in\cup_k \AA_k$ and $A'\in\AA'_n$ is
countable, and hence it follows that with probability 1,
\eqref{EQ:UNW} holds for all $n\ge 1$, every $A\in\cup_k \AA_k$ and
$A'\in\AA'_n$. Since $\cup_k \AA_k$ is dense in $\AA$, this implies
that \eqref{EQ:UNW} holds for all $n\ge 1$, every $A\in\AA$ and
$A'\in\AA'_n$.

From now on, we suppose that the choice of the $Y_i$ is such that
this holds.

For a fixed $n$, the indices $m$ have a subsequence $m_1<m_2<\dots$
such that $U_{n,m_j}$ converges to some function $U_n$ in the
weak-$*$-topology of $L_\infty(\AA_0 \times \AA'_n)$. Hence by
\eqref{EQ:UNW},
\[
\int\limits_{A\times A'} U_n\,d\pi\,d\pi' =\lim_{j\to\infty}
\int\limits_{A\times A'} U_{n,m_j}\,d\pi\,d\pi' =
\int\limits_{A\times A'} W\,d\pi\,d\pi'
\]
for all $n\ge 1$, every $A\in\AA$ and $A'\in\AA'_n$. Thus $U_n$ is a
representative of $\E(W\mid\AA\times\AA'_n)$. Since $U_n$ is
$\AA_0\times \AA'_n$ measurable, it is also a representative of
$\E(W\mid\AA_0\times\AA'_n)$. This shows that for every $n\ge 1$ we
have
\begin{equation}\label{EQ:WN}
\E(W\mid\AA\times\AA'_n)=\E(W\mid\AA_0\times\AA'_n)
\end{equation}
almost everywhere.

By Levy's Upward Theorem, the left hand side of \eqref{EQ:WN} tends
to $\E(W\mid\AA\times\AA')=W$ almost everywhere. The right hand side
of \eqref{EQ:WN} tends to $\E(W\mid\AA_0\times\AA')$ almost
everywhere, so $W=\E(W\mid\AA_0\times\AA')$ almost everywhere, which
proves the Lemma.
\end{proof}

We formulate a couple of corollaries, the first of which is
immediate:

\begin{corollary}\label{L-MEAS-1}
Let $(\Omega,\AA,\pi)$ and $(\Omega',\AA',\pi')$ be probability
spaces, let $W:~\Omega\times\Omega'\to\R$ be a bounded function that
is measurable with respect to $\AA\times\AA'$, and let
$\AA_0\subset\AA$ be a sub-$\sigma$-algebra. If $W(\cdot,y)$ is
$\AA_0$-measurable for almost all $x\in\Omega$, then  $W$ is almost
$\AA_0\times\AA'$-measurable.
\end{corollary}

\begin{corollary}\label{L-MEAS-2}
Let $(\Omega,\AA,\pi,W)$ be a graphon, and let $X_1,X_2,\dots$ be
independent random points from $\Omega$. Let $\AA_0\subseteq \AA$ be
the (random) $\sigma$-algebra generated by the functions
$W(\cdot,X_k)$. Then with probability $1$, $W$ is almost
$\AA_0\times\AA_0$-measurable.
\end{corollary}

\begin{proof}
Let $\AA_1$ denote the $\sigma$-algebras generated by the functions
$W(\cdot,X_{2k})$. Clearly $\AA_1\subseteq\AA_0$. By Lemma
\ref{L-MEAS-2x}, $W$ is almost $\AA_1\times\AA$ measurable with
probability $1$, so we can change it on a set of measure $0$ to get
an $\AA_1\times\AA$ measurable function $W'$. Let $\AA_2'$ be the
$\sigma$-algebras generated by the functions $W'(X_{2k+1},\cdot)$.
Applying the lemma again, we get that $W'$ is almost
$\AA_1\times\AA_2'$ measurable. With probability $1$, each function
$W(X_{2k+1},\cdot)$ differs from $W'(X_{2k+1},\cdot)$ on a set of
measure $0$ only (since the $X_{2k+1}$ are independent of $\AA_1$),
and so $\AA_2'\subseteq \sigma(\AA_0)$. So $W'$ is $\AA_0\times
\sigma(\AA_0)$ measurable, which implies that $W'$, and hence $W$,
are almost $\AA_0\times\AA_0$ measurable.
\end{proof}

\subsection{Anchor Sequences}
\label{sec:anchor}

Let us consider the $\sigma$-algebra $\LL$ on $[0,1]^\N$ generated by
the sets $A_1\times A_2\times\dots$, where each $A_i$ is a Borel
subset of $[0,1]$ and only a finite number of factors $A_i$ are
different from $[0,1]$. Fix a graphon $H=(\Omega,\AA,\pi,W)$ with
$0\leq W\leq 1$. For every $\alpha\in\Omega^\N$, the map
$\Phi_\alpha:~\Omega \to [0,1]^\N$ defined by \eqref{PHI-DEF} is
measurable, and \eqref{LAMBDADEF} defines a probability measure on
$\LL$ with respect to which $\Phi_\alpha$ is measure preserving. Thus
\eqref{Rad-Nik-def} leads to a symmetric, $\LL\times\LL$-measurable
function $W_{\Phi_\alpha}:~[0,1]^\N \times[0,1]^\N \to[0,1]$ which we
denote by $W_\alpha$. We say that $\alpha\in\Omega^\N$ is {\it
regular} if $W=W_\alpha^{\Phi_\alpha}$ almost everywhere.

\begin{lemma}\label{WW}
Almost all $\alpha\in\Omega^\N$ are regular.
\end{lemma}

\begin{proof}
Let $\AA_\alpha$ denote the $\sigma$-algebra of subsets of $\Omega$
of the form $\Phi_\alpha^{-1}(A)$, where $A\in\LL$. Note that
$\AA_\alpha\subseteq\AA$ by the fact that  $\Phi_\alpha$ is
measurable. Further, almost by definition, $\AA_\alpha$ is the
smallest sub-$\sigma$-algebra of $\AA$ such that all the functions
$W(\cdot,\alpha_i)$ are measurable. As a consequence, we may apply
Lemma \ref{L-MEAS-2} to conclude that for almost all $\alpha$, $W$ is
almost $\AA_\alpha\times\AA_\alpha)$-measurable, which by
Lemma~\ref{lem:W-Phi-Phi} gives that $W=W_\alpha^{\Phi_\alpha}$
almost everywhere.
\end{proof}

Let $\LL_\alpha$ be the completion of $\LL$ with respect to
$\lambda_\alpha$.  Then $([0,1]^\N,\LL_\alpha,\lambda_\alpha)$ is a
complete, Polish space and hence Lebesgue, so $H_\alpha=([0,1]^\N,
\LL_\alpha, \lambda_{\alpha}, W_\alpha)$ defines a Lebesguian
graphon.

\begin{lemma}\label{ANCHOR-ISO}
Let $H$ be a twin free graphon with the Lebesgue property. If
$\alpha$ is regular, then $\Phi_\alpha$ is an isomorphism mod $0$ and
$H_\alpha\cong H$.
\end{lemma}

\begin{proof}
By \eqref{LAMBDADEF}, $\Phi_\alpha$ is a measure preserving map from
$(\Omega,\AA,\pi)$ into $([0,1]^\N, \LL, \lambda_\alpha)$. Since
$(\Omega,\AA,\pi)$ is complete, $\Phi_\alpha$ is measurable (and
measure preserving) from  $(\Omega,\AA,\pi)$ into $([0,1]^\N,
\LL_\alpha, \lambda_\alpha)$ as well.  By the definition of a regular
$\alpha$, $H_\alpha^{\phi_\alpha}=H$ almost everywhere, and by Lemma
\ref{ISO-1}, $\Phi_\alpha$ is an isomorphism mod $0$.
\end{proof}

\section{Coupling}

\subsection{Partially Labeled Graphs and Marginals}\label{GRAPHALG}

We recall some notions from \cite{FLS}. A {\it partially labeled
graph} is a finite graph in which some of the nodes are labeled by
different nonnegative integers. Two partially labeled graphs are {\it
isomorphic}, if there is a label-preserving isomorphism between them.
A {\it $k$-labeled graph} is a partially labeled graph with labels
$1,\dots,k$.

Let $F_1$ and $F_2$ be two partially labeled graphs. Their {\it
product} $F_1F_2$ is defined as follows: we take their disjoint
union, and then identify nodes with the same label (retaining the
labels, and any multiple edges which this might create). For two
unlabeled graphs, $F_1F_2$ is their disjoint union. Clearly this
multiplication is associative and commutative.

Let $H=(\Omega,\AA,\pi,W)$ be a graphon, and let
$\alpha=(a_0,a_1,\dots)$ be an infinite sequence of points in
$\Omega$. Let $F$ be a partially labeled graph with nodes
$V(F)=\{1,\dots,k\}$, where nodes $1,\dots,r$ are labeled by distinct
nonnegative integers $\ell_1,\dots,\ell_r$. Let $X_i=a_{\ell_i}$ for
$1\le i\le r$, and let $X_{r+1},\dots,X_k\in\Omega$ be independent
points from the distribution $\pi$. Define
\[
t_\alpha(F,H)= \E\Bigl(\prod_{ij\in E(F)}W(X_j,X_j)\Bigr).
\]
Of course, this value only depends on those elements of $\alpha$
whose subscripts occur as labels, and we'll sometimes omit the tail
of $\alpha$ if it contains no labels. For example, if $F$ is a
$2$-labeled triangle, then
\begin{align*}
t_{a_1a_2}(F,H)&= t_\alpha(F,W)= \E(W(a_1,a_2)W(a_2,X)W(a_1,X))\\
&=\int\limits\limits_\Omega W(a_1,a_2)W(a_2,x)W(a_1,x)\,d\pi(x).
\end{align*}

It is easy to see that if $F_1$ and $F_2$ are two $k$-labeled graphs,
then
\[
t(F_1F_2,H)=\int\limits_{\Omega^k} t_{x_1\dots x_k}(F_1,W)t_{x_1\dots
x_k}(F_2,W)\,d\pi(x_1)\dots d\pi(x_k).
\]

\subsection{Multiple Edges}\label{MULTIPLE}

\begin{lemma}\label{lem:MULTIPLE}
Let $H=(\Omega,\AA,\pi,W)$ and $H'=(\Omega',\AA',\pi',W')$ be two
graphons, and assume that $t(F,H)=t(F,H')$ for every simple graph
$F$. Then $t(F,H)=t(F,H')$ for every multigraph $F$.
\end{lemma}

\begin{proof}
We use induction on the number of parallel edges in $F$. Suppose that
$F$ has two nodes, say $i$ and $j$, connected by more than one edge.
Let $F_k$ denote the multigraph obtained from $F$ by subdividing one
of these edges by $k-1$ new nodes. Let $F'$ denote the multigraph
obtained by removing one copy of the edge $ij$. So $F_1=F$, but for
$k>1$, $F_k$ has fewer parallel edges than $F$, and so we may assume
that
\[
t(F_k,H)=t(F_k,H')
\]
holds for every $k\ge 2$. We consider all the multigraphs $F_k$ and
$F'$ as $2$-labeled graphs, with $i$ and $j$ labeled $1$ and $2$.

Since $F_k$ can be thought of as the product of $F'$ and a path
$P_{k+1}$ with $k+1$ nodes (the endpoints labeled), we can write
\[
t(F,H)=\int\limits_{\Omega^2} W(x,y)t_{xy}(F',H)\,d\pi(x)\,d\pi(y),
\]
and
\[
t(F_k,H)=\int\limits_{\Omega^2}
t_{xy}(P_{k+1},H)t_{xy}(F',H)\,d\pi(x)\,d\pi(y).
\]
The first factor inside the integral can be expressed as
\[
t_{xy}(P_{k+1},H)=\int\limits_{\Omega^k} W(x,x_1)\cdots
W(x_{k-1}y)\,d\pi(x_1)\dots d\pi(x_{k-1}),
\]
which we can recognize as $k$-th power of the kernel $W$ as an
integral operator.

At this point, it will be useful to assume that $H$ and $H'$ are
countably generated graphons (this can be done without loss of
generality by Lemma~\ref{lem:Red-to-Count}).
As a consequence, $W$ is an integral operator on the separable
Hilbert space $L_2(\Omega,\AA,\pi)$, and since $W$ is bounded, this
implies that $W$ is Hilbert-Schmidt and thus compact, which in turn
implies that $W$ has a spectral representation:
\begin{equation}\label{UEXPAND}
W(x,y) \sim \sum_{n=0}^\infty \lambda_n \varphi_n(x)\varphi_n(y).
\end{equation}
It follows that for every
$k\ge 2$,
\[
t_{xy}(P_{k+1},H) = \sum_{n=0}^\infty
\lambda_n^k\varphi_n(x)\varphi_n(y),
\]
and hence
\[
t(F_k,H)=\sum_{n=0}^\infty \lambda_n^k \int\limits_{\Omega^2}
\varphi_n(x)\varphi_n(y) t_{xy}(F',H)\,d\pi(x)\,d\pi(y).
\]
Similarly, let
\[
W'(x,y) \sim \sum_{i=0}^\infty \mu_n \psi_n(x)\psi_n(y)
\]
be the spectral representation of $W'$, then we get that for every
$k\ge 2$,
\begin{equation}\label{EXPEQ}
0=t(F_k,H)-t(F_k,H')=\sum_{n=0}^\infty a_n \lambda_n^k-b_n \mu_n^k,
\end{equation}
where
\[
a_n=\int\limits_{\Omega^2}  \varphi_n(x)\varphi_n(y)
t_{xy}(F',U)\,d\pi(x)\,d\pi(y)
\]
and
\[
b_n=\int\limits\limits_{(\Omega')^2} \psi_n(x)\psi_n(y)
t_{xy}(F',W)\,d\pi(x)\,d\pi(y)
\]
are independent of $k$. (The integrals exist since $t_{x,y}(F',H)$ is
a bounded function of $x$ and $y$.) It follows that in (\ref{EXPEQ})
everything must cancel, in other words, for every value $c$,
\[
\sum \{a_n:~\lambda_n=c\}=\sum \{b_n:~\mu_n=c\}
\]
(it is known that the sums on both sides have a finite number of
terms, since the multiplicities of the eigenvalues are finite).

Now while (\ref{UEXPAND}) may not be true with equality, the
``trace'' with any other kernel gives an equation; in particular,
\[
t(F,H)=\sum_{n=0}^\infty \lambda_n \int\limits_{\Omega^2}
W(x,y)t_{xy}(F',H)\,d\pi(x)\,d\pi(y) =\sum_{n=0}^\infty a_n
\lambda_n,
\]
and similarly
\[
t(F,H')=\sum_{n=0}^\infty b_n \mu_n,
\]
which shows that $t(F,H)=t(F,H')$ as claimed.
\end{proof}

It will be convenient to assume that $0\le W,W'\le 1$. If this does
not hold, we can apply a linear transformation to the values of the
functions, to get two functions $W_0$ and $W_0'$ with $0\le
W_0,W_0'\le 1$.
Expanding the product in the definition \eqref{T-DEF}, $t(F,W_0)$ can
be written as a linear combination of the values $t(F',W)$, where
$F'$ is a subgraph of $F$. Thus $t(F,W)=t(F,W')$ for every graph $F$
if and only if $t(F,W_0)=t(F,W'_0)$ for every graph $F$ (where
``graph'' could mean either simple graph or multigraph). So
\eqref{MAINCOND} holds for $W_0$ and $W'_0$ if and only if it holds
for $W$ and $W'$. If we prove that this implies
$(\Omega,\AA,\pi,W_0)\cong (\Omega',\AA',\pi',W'_0)$, then $H\cong
H'$ follows trivially.

\subsection{Coupling Anchor Sequences}
\label{sec:coupling}

Consider two graphons $H=(\Omega,\AA,\pi,W)$ and
$H'=(\Omega',\AA',\pi',W')$ satisfying the conditions in Theorem
\ref{thm:Main} (i) and $0\le W,W'\le 1$.
Given two ``anchor'' sequences $\alpha =(a_1,a_2,\dots)$ from
$\Omega$ and $\beta=(b_1,b_2,\dots)$ from $\Omega'$, let
$H_\alpha=([0,1]^\N,\LL_\alpha, \lambda_{\alpha}, W_\alpha)$ and
$H'_\beta=([0,1]^\N, \LL_\beta',\lambda_\beta', W_\beta')$. We would
like to select $\alpha$ and $\beta$ in such a way that
$\lambda_\alpha=\lambda_\beta'$ and $W_\alpha=W_\beta'$ almost
everywhere. This will complete the proof of the theorem. By Lemma
\ref{WW}, we can guarantee that both $\alpha$ and $\beta$ are regular
by selecting $a_1,a_2,\dots$ as well as $b_1,b_2\dots$ independently
and uniformly from $\pi$ and $\pi'$, respectively; however, the
equality of $W_\alpha$ and $W'_\beta$ will only be true if we couple
$\alpha$ and $\beta$ carefully.

The condition on the coupling is described in the following lemma.

\begin{lemma}\label{KEY}
Let $H=(\Omega,\AA,\pi,W)$ and $H'=(\Omega',\AA',\pi',W')$ be two
graphons, and let $\alpha= (a_1,a_2,\dots)$ and
$\beta=(b_1,b_2,\dots)$ be regular sequences for $H$ and $H'$,
respectively. Suppose that for every partially labeled multigraph
$F$,
\[
t_\alpha(F,H)=t_\beta(F,H').
\]
Then  $\lambda_{\alpha}=\lambda_{\beta}'$ and $W_\alpha=W'_\beta$
almost everywhere (with respect to $\lambda_{\alpha}
=\lambda_{\beta}'$).
\end{lemma}

\begin{proof}
First, we show that $\lambda_\alpha=\lambda_\beta'$. These
probability measures are defined on the $\sigma$-algebra $\LL$ as the
distribution measures of the random variables
$W(X,a_1),W(X,a_2),\dots)$ and $W'(Y,b_1),W'(Y,b_2),\dots)$, where
$X$ and $Y$ are random points from $\pi$ and $\pi'$, respectively. By
Lemma~\ref{MOMENTS} it therefore suffices to prove that these random
variables have the same mixed moments.

Let $(k_1,k_2,\dots)$ be a sequence of nonnegative integers, of which
only a finite number is nonzero; say $k_i=0$ for $i>m$. Then
\[
\E(\prod_i W(X, a_i)^{k_i}) = t_{\alpha}(F,H),
\]
where $F$ is the star on $m+1$ nodes, with the endnodes labeled
$1,\dots,m$, and the edge between the center and endnode $i$ replaced
by $k_i$ parallel edges. Similarly,
\[
\E(\prod_i W'(Y, b_i)^{k_i}) = t_{\beta}(F,H').
\]
These numbers are equal by the hypothesis of the Lemma. This proves
that $\lambda_\alpha=\lambda_\beta'$.

Second, we show that $W_\alpha(x,y)=W'_\beta(x,y)$ for almost all
$x,y\in[0,1]^\N$. It suffices to show that the random variables
$Z_1=(X,Y,W_\alpha(X,Y))$ and $Z_2=(X,Y,W'_\beta(X,Y))$  (with values
from $[0,1]^\N\times [0,1]^\N\times[0,1]$) have the same
distribution, where $X$ and $Y$ are independent points in
$(\Omega^\N, \lambda_\alpha)$.

We can generate $Z_1$ by choosing independent uniform random points
$X'$ and $Y'$ from $\Omega$, and letting $X=\Phi_\alpha(X')$ and
$Y=\Phi_\beta(Y')$. Since $\alpha$ is regular, we have that
\[
W_\alpha(X,Y)=W(X',Y')
\]
with probability one, and hence
\[
Z_1=(W(X',a_1),W(X',a_2),\dots,W(Y',a_1),W(Y',a_2),\dots,W(X',Y')).
\]
Similarly, we have
\[
Z_2=(W'(X'',b_1),W'(X'',b_2),\dots,W'(Y'',b_1),W'(Y'',b_2),\dots,W'(X'',Y'')),
\]
where $X''$ and $Y''$ are independent random points from $\pi'$. To
prove that $Z_1$ and $Z_2$ have the same distribution, it again
suffices to prove that they have the same mixed moments.

A particular mixed moment is given by nonnegative integers
$(k_1,k_2,\dots)$, $(l_1,l_2,\dots)$ and $m$ (of which only a finite
number is nonzero; say $k_i=l_i=0$ for $i>n$). Let us define the
multigraph $F$ as follows. $F$ has two unlabeled nodes $v_x$ and
$v_y$, and $n$ further nodes labeled $1,\dots,n$. We connect $v_x$ to
$i$ by $k_i$ edges, $v_y$ to $i$ by $l_i$ edges ($i=1,\dots,n$), and
$v_x$ to $v_y$ by $m$ edges. Then
\[
\E\bigl(W(X',a_1)^{k_1}\cdots W(X',a_n)^{k_n}W(Y',a_1)^{k_1}\cdots
W(Y',a_n)^{k_n}W(X',Y')\bigr)=t_\alpha(F,H).
\]
and similarly
\[
\E\bigl(W'(X'',b_1)^{k_1}\cdots
W'(X'',b_n)^{k_n}W'(Y'',b_1)^{k_1}\cdots
W'(Y'',b_n)^{k_n}W'(X,Y)\bigr)=t_\beta(F,H').
\]
These two numbers are the same by hypothesis. This completes the
proof of the Lemma.
\end{proof}

To prove Theorem \ref{thm:Main}, we next show:

\begin{lemma}\label{LEM:COUPLE}
Let $H=(\Omega,\AA,\pi,W)$ and $H'=(\Omega',\AA',\pi',W')$ be two
Lebesguian graphons such that
\[
t(F,H)=t(F,H').
\]
for every multigraph $F$. Then we can couple sequences
$\alpha\in\Omega^\N$ with sequences $\beta\in{\Omega'}^\N$ so that if
$\alpha,\beta)$ is a sequence from this joint distribution, then
\[
t_\alpha(F,H)=t_\beta(F,H').
\]
holds almost surely for every partially labeled multigraph $F$.
\end{lemma}

\begin{proof}
Let $\FF_k$ be the set of $k$-labeled multigraphs. We define
recursively a coupling of sequences $\alpha\in\Omega^k$ with
sequences $\beta\in{\Omega'}^k$ so that
$t_{\alpha'}(F,H)=t_{\beta'}(F,H')$ holds almost surely for every
$F\in\FF_k$. Let $(a_1,\dots,a_k)$ and $(b_1,\dots,b_k)$ be chosen
from this coupled distribution. Consider two random points $X$ from
$\pi$ and $Y$ from $\pi'$, and the random variables
\[
A=(t_{a_1\dots a_kX}(F,H):~F\in\FF_{k+1})
\]
and
\[
B= (t_{b_1\dots b_kY}(F,H):~F\in\FF_{k+1})
\]
with values in $[0,1]^{\FF_{k+1}}$. We claim that the variables $A$
and $B$ have the same distribution. It suffices to show that $A$ and
$B$ have the same mixed moments. Consider any moment of $A$; in other
words, let $F_1,\dots,F_m\in\FF_{k+1}$,  let $q_1,\dots,q_m$ be
nonnegative integers, and let $F_i^{q_i}$ be obtained from $F_i$ by
replacing each edge in $F_i$ by $q_i$ edges. Then the corresponding
moment of $A$ is
\[
\E\Bigl(\prod_{i=1}^m t_{a_1\dots a_kX}(F_i,H)^{q_i}\Bigr)
=\E\bigl(t_{a_1\dots a_k X}(F_1^{q_1}\dots F_m^{q_m},H)\bigr)
=t_{a_1\dots a_k}(F,H),
\]
where the multigraph $F$ is obtained by unlabeling the node labeled
$k+1$ in the multigraph $F_1^{q_1}\dots F_m^{q_m}$. Expressing the
moments of $B$ in a similar way, we see that they are equal by the
induction hypothesis. This proves that $A$ and $B$ have the same
distribution.

Using Lemma~\ref{COUPLE-3}
it follows that we can couple the variables $X$ and $Y$ so that $A=B$
with probability 1. In other words, we can replace $X$ and $Y$ by a
random variable $(X',Y')\in\Omega\times\Omega'$ so that $X'$ has
distribution $\pi$, $Y'$ has distribution $\pi'$, and their joint
distribution satisfies
\[
t_{a_1\dots a_kX'}(F,H)= t_{b_1\dots b_kY'}(F,H')
\]
for every $F\in\FF_{k+1}$ with probability 1. Thus we have extended
the coupling to $\Omega^k\times{\Omega'}^k$.

It is clear that this sequence of couplings defines a coupling of
$\Omega^\N$ with ${\Omega'}^\N$ as claimed.
\end{proof}

\subsection{Conclusion of proofs}

\noindent{\bf Proof of Theorem \ref{thm:Main}.} Part (i) follows
easily: if we choose random sequences $(\alpha,\beta)$ from the
coupled distribution given by Lemma \ref{LEM:COUPLE}, then these
sequences will be regular with probability $1$, and so they satisfy
the conditions of Lemma \ref{KEY}.

To prove (ii), suppose that $H=(\Omega,\AA,\pi,W)$ and
$H'=(\Omega',\AA',\pi',W')$ satisfy \eqref{MAINCOND} for every simple
graph $F$. By Corollary \ref{COR:ON-LEBESGUE}, we can find twin-free
Lebesguian graphons $G=(\Gamma,\BB,\rho,U)$ and
$G'=(\Gamma',\BB',\rho',U')$ and weak isomorphisms $\phi$ and $\phi'$
from $H$ and $H'$ to $G$ and $G'$, respectively. It follows by
Theorem \ref{thm:Main}(i) that the ${G}$ and ${G'}$ are isomorphic
mod $0$, so in particular $U=(U')^{\psi'}$ almost everywhere for some
measure preserving map $\psi':\Gamma\to\Gamma'$. Defining
$\psi:\Omega\to\Gamma'$ by $\psi(x)=\psi'(\phi(x))$, we conclude that
$W=(U')^\psi$ almost everywhere. The maps $\psi$ and $\phi'$ are
measure preserving from the completions $\overline{H}$ and
$\overline{H'}$ into $G'$.\hfill$\square$\medskip

\noindent{\bf Proof of Corollary \ref{cor:UNIQUENESS}.} The
equivalence of (a), (b) and (c) follows by Theorem~\ref{thm:Main}
(ii) and the fact that a function which is measurable with respect to
the completion of $\LL\times\LL$ is almost everywhere equal to a
function which is measurable with respect to $\LL\times\LL$. In the
proof of (c), Theorem~\ref{thm:Main} may give a graphon containing
atoms, but it is easy to replace these atoms by intervals of
appropriate length.

To prove that (c)$\Longrightarrow$(e), assume that $\varphi,\psi$ and
$U$ exist as in (c). Let $X,X'\in[0,1]$ be independent random points
from the uniform distribution $\lambda$ on $[0,1]$. Since $\varphi$
and $\psi$ are measure preserving, $\varphi(X)$ and $\psi(Y)$ have
the same distribution, and hence by Lemma \ref{COUPLE-3}
there is a coupling measure $\gamma$ on $[0,1]\times[0,1]$ with
marginals $\lambda$ such that if $(X,X')$ is a random sample from
$\gamma$, then $\varphi(X)=\psi(X')$ with probability $1$. So if
$(X,X')$ and $(Y,Y')$ are independent random points from $\gamma$,
then
\[
W(X,Y) = U(phi(X),phi(Y)) = U'(psi(X'),psi(Y')) = W'(X',Y').
\]

To prove that (e)$\Longrightarrow$(d), consider the projections
$\Phi,\Psi:~[0,1]^2\to[0,1]$ defined by $\Phi(x,x')=x$ and
$\Psi(x,x')=x'$. Then
\[
W^{\Phi}((X,X'),(Y,Y'))=W(X,Y)
\]
and
\[
(W')^{\Psi}((X,X'),(Y,Y'))=W'(X',Y')
\]
Thus, $W^{\Phi}=(W')^{\Psi}$ almost everywhere. Furthermore, $\Phi$
and $\Psi$ are measure preserving if we consider the coupling measure
$\gamma$ on $[0,1]$.

Since the completion of $([0,1]^2,\LL_2,\gamma)$ is a Lebesgue space,
we can find a measure preserving map
$\rho:~([0,1],\lambda)\to([0,1]^2,\gamma)$. Setting
$\varphi=\Phi\circ \rho$ and $\psi=\Psi\circ\rho$, we obtain the
desired measure preserving maps $\varphi,\psi:[0,1]\to[0,1]$ such
that $W^\varphi=(W')^\psi$ almost everywhere.

Finally, (d)$\Rightarrow$(a) is trivial.\hfill$\square$

\section*{Acknowledgement}

We are grateful to Mikl\'os Laczkovich, Ron Peled, Yuval Peres and
Oded Schramm for many useful discussions on the topic of this paper,
and to Kati Vesztergombi and Svante Janson for carefully reading an
earlier version and suggesting several improvements.

\section{Appendix: Moments and coupling of probability distributions}

In this section we prove some probability theory lemmas, that are
``well known'' but not easy to reference. We start with the fact that
if two vector valued random variables have the same mixed moments,
then they have the same distribution (cf Feller \cite{Fel}, Problem
XV.9.21.).

\begin{lemma}\label{MOMENTS}
Let $(\Omega,\AA,\pi)$ and $(\Omega',\AA',\pi')$ be probability
spaces, and let $f:~\Omega\to [0,1]^\N$ and $g:~\Omega'\to[0,1]^\N$
be measurable functions, with $f(x)=(f_1(x),f_2(x),\dots)$ and
$g(y)=(g_1(y),g_2(y),\dots)$. If
\[
\int f_1(x)^{k_1}\dots f_n(x)^{k_n}\, d\pi(x) = \int
g_1(y)^{k_1}\dots g_n(y)^{k_n}\, d\pi'(y)
\]
for every finite sequence of nonnegative integers $k_1,\dots,k_n$,
then $\pi(f^{-1}(B))=\pi'(g^{-1}(B))$ for every Borel set $B\subseteq
[0,1]^\N$.
\end{lemma}

\begin{proof}
It suffices to prove that $\pi(f^{-1}(B))=\pi'(g^{-1}(B))$ for every
Borel set of the form $B=I_1\times I_2\times \dots I_n\times
[0,1]\times\dots$, where $I_1,\dots,I_n$ are intervals. Let
$p_{j,m}(x)$ be a polynomial that approximates the indicator function
$\one_{I_j}$ on $[0,1]$ in $L_1$ with error less than $1/m$
$(j=1,\dots,n)$.
Then
\begin{align*}
\int\limits_{\Omega} p_{1,m}(f_1(x))\cdots p_{n,m}(f_n(x))\,dx
&\longrightarrow \int\limits_{\Omega} \one_{I_1}(f_1(x))\cdots
\one_{I_n}(f_n(x))\,dx\\
&= \int\limits_{f^{-1}(B)} 1\,dx =\pi(f^{-1}(B)) \qquad (m\to\infty).
\end{align*}
Similarly,
\[
\int\limits_{\Omega} p_{1,m}(g_1(x))\cdots p_{n,m}(g_n(x))\,dx
\longrightarrow \pi(g^{-1}(B)) \qquad (m\to\infty).
\]
But the left hand sides of these two relations are equal for all $m$,
which proves the Lemma.
\end{proof}

We need the following natural fact about coupling.

\begin{lemma}\label{COUPLE-3}
Assume that $(\Omega,\AA,\pi)$ and $(\Omega',\AA,\pi')$ are Lebesgue
spaces, and $(\Gamma,\BB,\rho)$, a countably generated separating
space. Let $f:~\Omega\to\Gamma$ and $g:~\Omega'\to\Gamma$ be measure
preserving maps. Then there exists a coupling $\nu$ of
$(\Omega,\AA,\pi)$ and $(\Omega',\AA,\pi')$ such that
\[
\nu\bigl\{(x,y):~f(x)=g(y)\bigr\}=1.
\]
\end{lemma}

\begin{proof}
For $A\in\AA$, consider the measure $\lambda^A(B)=\pi(A\cap
f^{-1}(B))$ defined for $B\in\BB$, and its Radon-Nikodym derivative
$f^A=d\lambda^A/d\rho$. Since $\lambda^A\le\pi(f^{-1}(B))=\rho(B)$,
this derivative exists, and $0\le f^A\le 1$ almost everywhere.
Furthermore, $f^\emptyset=0$ and $f^\Omega=1$ almost everywhere.

\killtext{It is also easy to check that
\begin{equation}\label{FFB}
f^{f^{-1}(B)\cap A}=f^{A}\one_B
\end{equation}
almost everywhere for all $B\in\BB$ and $A\in\AA$. From \eqref{FFB}
it is easy to derive that
\begin{equation}\label{NUBC}
\nu(f^{-1}(B)\times C)= \pi'(f^{-1}(B)\cap C),
\end{equation}
and in particular,
\begin{equation}\label{NUBB}
\nu(f^{-1}(B)\times g^{-1}(B))= \rho(B).
\end{equation}
}

Similarly, for $C\in\AA'$, define $\mu^C(B)=\pi'(C\cap g^{-1}(B))$
and $g^C=d\mu^C/d\rho$. Finally, let
\begin{equation}\label{COMP-DEF}
\nu(A\times C)=\int f^Ag^C\,d\rho.
\end{equation}
Clearly
\[
\nu(A\times C)\le \int f^A\,d\rho = \pi(A),
\]
and similarly $\nu(A\times C)\le \pi'(C)$. Hence in particular
$\nu(A\times C)=0$ if either $\pi(A)=0$ or $\pi'(C)=0$.

\begin{claim}\label{CL:ADD}
If $A_i\in\AA$, $C_i\in\AA'$ ($i\in I$) and the sets $A_i\times C_i$
form a (finite or countably infinite) partition of $A\times C$
($A\in\AA$, $C\in\AA'$), then $\sum_i \nu(A_i\times C_i)=\nu(A\times
C)$.
\end{claim}

It is easy to see that if $A_1,A_2\in\AA$ are disjoint sets and
$A=A_i\cup A_2$, then $f^{A_1}+f^{A_2}=f^A$ almost everywhere. It
follows that for every $C\in\AA'$, we have $\nu(A_1\times
C)+\nu(A_2\times C)=\nu(A\times C)$. This implies by standard
arguments that the claim holds if $|I|$ is finite. This in turn
implies that $\nu$ extends to a finitely additive measure on the
algebra $\FF$ of sets that can be written as the union of a finite
number of product sets $A\times C$ ($A\in\AA$, $C\in\AA'$).

In the case of infinite $|I|$, it follows that $\sum_i \nu(A_i\times
C_i)\le \nu(A\times C)$; in fact, for every finite $J\subseteq I$, we
have $\cup_{i\in J} A_i\times C_i \subseteq A\times C$, and hence by
the finite additivity of $\nu$, we have
\[
\sum_{i\in J}\nu(A_i\times C_i)=\nu\Bigl(\cup_{i\in J} A_i\times
C_i\Bigr) \le \nu(A\times C).
\]
Since this holds for every finite subset $J$ of $I$, it also holds
for $I$.

Suppose that there is a partition where $\{A_i\times C_i:~i=1\in\N\}$
of $A\times C$ and an $\eps>0$ for which
\[
\sum_i \nu(A_i\times C_i) < \nu(A\times C)-4\eps
\]
on a set $B$ of positive measure. Now we use that $(\Omega,\AA,\pi)$
and $(\Omega',\AA,\pi')$ are Lebesgue spaces, so we may assume that
they are intervals $[0,a]$ and $[0,b]$ respectively, together with a
countable set of atoms. Thinking of the atoms as converging to $a$
from above, we have a compact topology on them. For every $i$, we can
find an open sets $U_i\supseteq A_i$ and $V_i\supseteq C_i$ such that
$\pi(U_i)\le\pi(A_i)+\eps 2^{-i}$ and
$\pi'(V_i)\le\pi'(C_i)+\eps/2^i$. Also, we can find closed sets
$U\subseteq A$ and $V\subseteq C$ such that $\pi(U)\ge\pi(A)-\eps$
and $\pi'(V)\ge\pi'(C)-\eps$. Then
\begin{align*}
\nu(U_i\times V_i) &\le \nu(A_i\times C_i)+\nu((U_i\setminus
A_i)\times C_i) + \nu(U_i\times (V_i\setminus C_i))\\
&\le \nu(A_i\times C_i) +\pi(U_i\setminus A_i)+\pi'(V_i\setminus
C_i)\le\nu(A_i\times C_i) +2\eps 2^{-i}.
\end{align*}
It follows similarly that
\[
\nu(U\times V) \ge\nu(A\times C) -2\eps.
\]
Hence
\[
\sum_i \nu(U_i\times V_i) \le \sum_i\nu(A_i\times C_i)+2\eps <
\nu(A\times C)-2\eps\le\nu(U\times V).
\]
The open sets $U_i\times V_i$ cover the compact set $U\times V$, and
so a finite number of them also covers. But the contradicts the
finite additivity of $\nu$ which we already established.

\begin{claim}\label{CL:MEAS}
The setfunction $\nu$ extends to a measure on $\AA\times\AA'$.
\end{claim}

We have seen already that $\nu$ extends to $\FF$; it follows by Claim
\ref{CL:ADD} that this extension is $\sigma$-additive. Thus the Claim
follows by the Measure Extension Theorem.

Define $\Delta=\{(x,y)\in \Omega\times\Omega': f(x)=g(y)\}$. To
complete the proof of the Lemma, we want to prove that $\nu$ is a
coupling between $(\Omega,\AA,\pi)$ and $(\Omega',\AA,\pi')$ (which
is trivial), and that $\nu(\Omega\times\Omega'\setminus \Delta)=0$.
Let $\SS\subseteq\BB$ be a countable family separating the elements
of $\Gamma$. Then
\[
\Omega\times\Omega'\setminus \Delta = \bigcup_{S\in\SS}
f^{-1}(S)\times g^{-1}(\Gamma\setminus S) \cup\bigcup_{S\in\SS}
f^{-1}(\Gamma\setminus S)\times g^{-1}(S).
\]
Consider any term here, say $f^{-1}(S)\times g^{-1}(\Gamma\setminus
S) = A\times C$. Then
\[
\nu(A\times C)=\int f^A g^C\,d\rho = \int_S + \int_{\Gamma\setminus
S}.
\]
Here
\[
\int_S f^A g^C\,d\rho \le \int_S g^C\,d\rho =
\mu^C(S)=\pi'(g^{-1}(\Gamma\setminus S) \cap g^{-1}(S))=0,
\]
and similarly
\[
\int_{\Omega\setminus S} f^Ag^C\,d\rho = 0.
\]
This proves that $\nu(\Omega\times\Omega'\setminus \Delta)=0$.
\end{proof}


\begin{thebibliography}{99}

\bibitem{AFKK}
N.~Alon, W.~Fernandez de la Vega, R.~Kannan and M.~Karpinski: Random
sampling and approximation of MAX-CSPs, {\it J. Comput. System Sci.}
{\bf 67} (2003), 212--243.

\bibitem{BCLSV-unpub}
C.~Borgs, J.~Chayes, L.~Lov\'asz, V.T.~S\'os, K.~Vesztergombi,
unpublished, 2004.

\bibitem{BCLSV-rev}
C.~Borgs, J.~Chayes, L.~Lov\'asz, V.T.~S\'os, K.~Vesztergombi:
Counting graph homomorphisms, in: {\it Topic in Discrete Mathematics}
(Klazar, Kratochvil, Loebl, Matousek, Thomas, Valtr, eds.) Springer,
Berlin--Heidelberg (2006), 315--371.

\bibitem{dense1}
C.~Borgs, J.~Chayes, L.~Lov\'asz, V.T.~S\'os, K.~Vesztergombi:
Convergent Sequences of Dense Graphs I: Subgraph Frequencies, Metric
Properties and Testing, preprint (2006),
\\
\url{http://research.microsoft.com/~borgs/Papers/ConvMetric.pdf}

\bibitem{dense2}
C.~Borgs, J.T.~Chayes, L.~Lov\'asz, V.T.~S\'os, and K.~Vesztergombi:
Convergent Graph Sequences II. $H$-Colorings,
Statistical Physics and Quotients, manuscript (2006)

\bibitem{ELS}
P.~Erd\"os, L.~Lov\'asz, J.~Spencer: Strong independence of graphcopy
functions, in: {\it Graph Theory and Related Topics}, Academic Press,
165-172.

\bibitem{Fel}
Feller, W.: {\it An Introduction to Probability Theory and its
Applications,} Vol. II, Second edition, Wiley, NewYork, (1971).

\bibitem{FLS}
M.~Freedman, L.~Lov\'asz, A.~Schrijver: Reflection positivity, rank
connectivity, and homomorphism of graphs {\it Journal of The American
Mathematical Society} (to appear)

\bibitem{GGR}
O.~Goldreich, S.~Goldwasser and D.~Ron: Property testing and its
connection to learning and approximation, {\it J. ACM} {\bf 45}
(1998), 653--750.

\bibitem{Hal}
P.R.~Halmos: {\it Measure Theory}, Graduate Texts in Mathematics
{\bf 18}, Springer, New York, Heidelberg, Berlin (1991).

\bibitem{LSz}
L.~Lov\'asz and B.~Szegedy: Limits of dense graph sequences,
Microsoft Research Technical Report TR-2004-79.

\bibitem{LSz2}
L.~Lov\'asz, B.~Szegedy: Szemer\'edi's Lemma for the analyst,
Microsoft Research Technical Report MSR-TR-2005-90,

\url{ftp://ftp.research.microsoft.com/pub/tr/TR-2005-90.pdf}

\bibitem{Roh}  V. A. Rohlin: On the fundamental ideas of measure theory,
Translations of the American Mathematical Society, Series 1, Vol. 10, 1-54 (1962).
The Russian original appeared in Mat. Sb. 25, 107-150 (1949).

\bibitem{Wil}
D.~Williams: {\it Probability with Martingales}, Cambridge University
Press, (1991).

\end{thebibliography}
\end{document}